\documentclass[11pt,reqno]{amsart}

\usepackage{amsmath, amsthm, amssymb}
\usepackage{setspace}
\usepackage{mathtools}
\usepackage{geometry}
\usepackage{hyperref}
\usepackage{enumitem}
\usepackage{graphicx}

\title{Finiteness of Disjoint Covering Systems with Precisely One Repeated Modulus}

\author{Yu Hashimoto}
\address{Independent Researcher, Tokyo, Japan}
\email{hashimoto.peace@gmail.com}
\keywords{disjoint covering systems, covering systems, arithmetic progressions, coset partitions, residue systems, combinatorial number theory}

\newtheorem{theorem}{Theorem}[section]
\newtheorem{lemma}[theorem]{Lemma}
\newtheorem{corollary}[theorem]{Corollary}
\newtheorem{proposition}[theorem]{Proposition}
\theoremstyle{definition}
\newtheorem{definition}[theorem]{Definition}
\newtheorem{remark}[theorem]{Remark}
\newtheorem{problem}[theorem]{Problem}
\newtheorem{question}[theorem]{Question}
\newtheorem{conjecture}[theorem]{Conjecture}

\newcommand{\N}{\mathbb{N}}
\newcommand{\Z}{\mathbb{Z}}

\DeclareMathOperator{\ind}{index}

\newcommand{\Cset}{\mathcal{C}}
\newcommand{\Aset}{\mathcal{A}}
\newcommand{\Lset}{\mathcal{L}}
\newcommand{\Dset}{\mathcal{D}}

\mathtoolsset{showonlyrefs=true}

\begin{document}

\begin{abstract}
We prove that for each fixed \(m \ge 2\), there are only finitely many disjoint covering systems with minimum modulus at least \(3\) in which precisely one modulus is repeated, namely the largest modulus, and it occurs exactly \(m\) times.
\end{abstract}
\maketitle
\section{Introduction}

A collection of arithmetic progressions \((a_i+n_i\Z)_{i=1}^t\) is called a \emph{disjoint covering system} (DCS) if
every integer lies in exactly one of these progressions.
It is a classical theorem of Mirsky, Newman, Davenport, and Rado (see \cite{Erdos}) that the moduli of a DCS cannot all be distinct: the largest modulus must occur at least twice.
Motivated by this, we study DCSs in which \emph{only} the largest modulus is repeated, namely those satisfying
\begin{equation}\label{eq:distinctas}
n_1 < n_2 < \cdots < n_{t-m} < n_{t-m+1}=\cdots=n_t
\qquad (m\ge 2).
\end{equation}
We call a DCS \emph{trivial} if \(m=t\).
There exist infinitely many systems satisfying \eqref{eq:distinctas} for \(m = 2\), namely
\begin{equation}\label{eq:trivial}
(2\Z, 4\Z +1, \ldots ,2^{t-1}\Z + 2^{t-2} - 1,2^{t-1}\Z + 2^{t-1} - 1).
\end{equation}
Following \cite{BFF}, we may normalize \eqref{eq:distinctas} using the \emph{2-add}/\emph{2-drop} procedure and thus restrict
attention to the case where the smallest modulus is at least \(3\):
\begin{equation}\label{eq:basicas}
3 \le n_1 < n_2 < \cdots < n_{t-m} < n_{t-m+1}=\cdots=n_t
\qquad (m\ge 2).
\end{equation}
Indeed, if \(n_1 = 2\), then disjointness forces all other moduli to be even.
In this situation we apply the \emph{2-drop} operation: remove \(n_1\) and divide all remaining moduli by \(2\). Iterating 2-drop, we either arrive at the trivial system \(( 2\Z,2\Z + 1 )\)
or obtain a system with \(n_1 \ge 3\).
Conversely, any system satisfying \eqref{eq:basicas} gives rise to systems with \(n_1 = 2\) via repeated \emph{2-add},
which adjoins the modulus \(2\) and doubles all other moduli. Thus, we may assume \eqref{eq:basicas} without loss of generality.

Interpreting the earlier results via the above normalization \eqref{eq:basicas}, one obtains the following:
Stein \cite{Stein} showed that there is no nontrivial DCS with \(m=2\), and Zn\'am \cite{Znam1969,Znam1984} proved the same for \(m=3\).
Porubsk\'y \cite{Porubsky} determined the unique nontrivial DCS for \(m=4\) and showed that for \(m=5\) there is no nontrivial system.
Berger, Felzenbaum, and Fraenkel \cite{BFF} classified all nontrivial DCSs satisfying \eqref{eq:basicas} for \(4\le m\le 9\). Zeleke and Simpson \cite{ZS} extended this classification to \(10\le m\le 12\), and Ekhad, Fraenkel, and Zeilberger \cite{EFZ} enumerated examples up to \(m=32\).
Although they did not prove completeness, they verified their list and exhausted all systems with largest modulus \(\le 600\).
They further conjectured that, for each fixed \(m\), only finitely many systems exist in the normalized setting.

Motivated by these results, we prove the following finiteness statement.

\begin{theorem}\label{thm:main-dcs}
For each fixed \(m\ge 2\), there exist only finitely many DCSs whose moduli satisfy
\eqref{eq:basicas}.
\end{theorem}

To prove Theorem~\ref{thm:main-dcs}, it is convenient to pass to an equivalent formulation in terms of coset partitions.
Let \(n:=\mathrm{lcm}(n_1,\dots,n_t)\). Reducing modulo \(n\), each progression \(a_i+n_i\Z\) corresponds to a coset
\(K_i\subset \Z_n\) of the subgroup of index \(n_i\), hence \(|K_i|=n/n_i\). Moreover, it is known that \eqref{eq:basicas} implies \(n_i \mid n_t\) for all \( 1\le i \le t\) (see \cite{BFF}), hence \(n_t = n\). In this way, determining DCSs satisfying
\eqref{eq:basicas} is equivalent to determining coset partitions
\[
\Cset=(K_1,\ldots,K_t)
\]
of the additive cyclic group \(\Z_n\) satisfying
\begin{equation}\label{eq:basicascoset}
\frac{n}{3} \ge |K_1| > |K_2| > \cdots > |K_{t-m}| > |K_{t-m+1}|=\cdots=|K_t| = 1.
\end{equation}
In particular, the repeated largest modulus corresponds to the \(m\) singleton cosets in \eqref{eq:basicascoset}.

Our main result is most naturally stated in the coset-partition setting.

\begin{theorem}\label{thm:main-coset}
Let \(\Cset=(K_1,\ldots,K_t)\) be a coset partition of \(\Z_n\) satisfying \eqref{eq:basicascoset}.
Then the integer \(n\) is bounded by a constant depending only on \(m\).
\end{theorem}

There is another aspect of Theorem~\ref{thm:main-coset}, which connects the disjoint setting to quantitative
lower bounds for uncovered sets in distinct-moduli covering problems.

\begin{remark}\label{rem:balister-simpson}
Let \((a_i+n_i\Z)_{i=1}^t\) be a DCS satisfying \eqref{eq:basicas} and write \(n:= \mathrm{lcm}(n_1,\dots,n_t)\).
Removing the \(m\) progressions with the repeated largest modulus \(n\) leaves a family of arithmetic progressions with
\emph{distinct} moduli \(n_1,\dots,n_{t-m}\) whose uncovered set
\[
R:=\Z\setminus \bigcup_{i=1}^{t-m} (a_i+n_i\Z)
\]
has asymptotic density
\[
\delta(R)=\frac{m}{n}.
\]

Balister, Bollobás, Morris, Sahasrabudhe, and Tiba \cite{BBMST} proved that for every \(\varepsilon>0\) there exists \(M=M(\varepsilon)\) such that
for any collection of arithmetic progressions with \emph{distinct} moduli \(d_1,\dots,d_k\ge M\), the uncovered set has density
at least \(\tfrac12 e^{-4C}\), where \(C=\sum_{i=1}^k \mu(d_i)/d_i\) for an explicit multiplicative function \(\mu\)
(depending on \(\varepsilon\)).
Applying this to the distinct-moduli system above in the additional regime \(n_1\ge M\) yields
\[
\frac{m}{n}=\delta(R)\ \ge\ \frac12 e^{-4C},
\qquad\text{hence}\qquad
n\ \le\ 2m\,e^{4C}.
\]

To compare with our finiteness theorem, note that Simpson \cite{Simpson} proved that in a DCS in which
each modulus occurs at most \(m\) times,
\[
m\prod_{p < p^*} \frac{p}{p-1} \ge p^* 
\]
holds where \(p^*\) is the greatest prime dividing any modulus. It follows that the primes dividing the moduli \(n_1,\dots,n_{t-m}\) are all \(\le p^*(m)\) for some function \(p^*(m)\) (which we will recover later in our specific setting \eqref{eq:basicas} in Lemma~\ref{lem:bound-vh-general}) by Mertens’ theorem
\[
\prod_{p < p^*} \frac{p}{p-1} \sim e^\gamma \log p^* 
\]
where \(\gamma\) is Euler's constant.
Since the multiplicative function \(\mu\) in \cite{BBMST} satisfies \(\mu(p^i) = \mu(p)\), the quantity 
\(C\) is then bounded in terms of \(m\) as well.
Consequently, in the large-minimum-modulus regime \(n_1\ge M\) one obtains a bound of the form \(n\le N(m)\) from the density estimate of Balister et al. together with Simpson's prime-factor bound.

Theorem~\ref{thm:main-coset} shows that, in the disjoint setting, this boundedness already holds under the minimal
normalization \(n_1\ge 3\), i.e.\ without any largeness assumption on the moduli.
\end{remark}

\section{Parallelotope partitions and cell inequalities}\label{sec:parallelotope}

This section recalls the ``parallelotope model'' of Berger, Felzenbaum, and Fraenkel~\cite{BFF}
and proves a slightly more flexible inequality that contains both Proposition~I in \cite[\S2]{BFF} and
a variant used later in this paper as immediate corollaries. Although we present the inequality in a slightly more flexible form, the underlying idea is essentially the same as in the proof of \cite{BFF}.

\subsection{Parallelotopes, cells, and cell partitions}

Fix an integer \(d\ge 1\) and a vector \(\mathbf{b}=(b_1,\dots,b_d)\) with \(b_i\in\N\) and \(b_i\ge 2\).
We consider the discrete box
\[
  \mathcal{P}(\mathbf{b})
  :=\prod_{i=1}^d \{0,1,\dots,b_i-1\}\subset \Z^d,
\]
and we call \(\mathcal{P}(\mathbf{b})\) the \emph{\((d;\mathbf{b})\)-parallelotope}.

Following~\cite{BFF}, for \(T\subseteq [d]:=\{1,\dots,d\}\) and \(u=(u_1,\dots,u_d)\in \mathcal{P}(\mathbf{b})\), we define the
\emph{\(T\)-cell} (or \emph{cell of index \(T\)}) by
\[
  X(u,T):=\{c=(c_1,\dots,c_d)\in \mathcal{P}(\mathbf{b}) : c_i=u_i\ \text{for all } i\notin T\}.
\]
We set \(\ind(X(u,T)):=T\). Note that coordinates in \(T\) are the \emph{free} coordinates, and those in \([d]\setminus T\) are fixed.
The cardinality of a cell depends only on its index:
\[
  |X(u,T)|=\prod_{i\in T} b_i.
\]
We may suppress \(u\) and \(T\) from the notation when they are clear.
A \emph{cell partition} of \(\mathcal{P}(\mathbf{b})\) is a finite family of cells
\[
  \mathcal{F}=(X_1,\dots,X_t)
\]
such that \(\mathcal{P}(\mathbf{b})=\bigsqcup_{j=1}^t X_j\). We always assume \(t>1\).

\medskip
We will repeatedly pass to coordinatewise ``shrunk'' sub-parallelotopes.
If \(\hat{\mathbf{b}}=(\hat b_1,\dots,\hat b_d)\) satisfies \(2 \le \hat b_i\le b_i\) for all \(i\), set
\[
  \widehat{\mathcal{P}}:=\mathcal{P}(\hat{\mathbf{b}})\subseteq \mathcal{P}(\mathbf{b}).
\]
For a cell \(X=X(u,T)\), the intersection \(\hat X:=X\cap \widehat{\mathcal{P}}\) is either empty or again a \(T\)-cell inside
\(\widehat{\mathcal{P}}\) (with the same index \(T\)). In particular, if we only shrink coordinates in \([d]\setminus T\),
then \(\hat X=X\) whenever \(\hat X\neq \emptyset\).
We freely identify \(\mathcal{P}(\mathbf{b})\) with the finite abelian group
\[
  G(\mathbf{b}) := \prod_{i=1}^d \Z_{b_i}
\]
via the obvious coordinatewise reduction map. In particular, translation by an element of \(G(\mathbf{b})\)
acts on \(\mathcal{P}(\mathbf{b})\), preserving cardinalities of subsets and the cell-partition structure.

\subsection{Cell inequalities}

The key observation in~\cite{BFF} is that after shrinking certain \emph{fixed} coordinates to a common modulus \(b\),
one can partition the shrunk parallelotope into \(b\) congruence classes and compare the total mass of cells lying in each class.

\begin{proposition}\label{prop:gen-cell-ineq}
Let \(\mathcal{F}=(X_1,\dots,X_t)\) be a cell partition of \(\mathcal{P}(\mathbf{b})\).
Fix a cell \(X\in\mathcal{F}\).
Let \(S\subseteq [d]\) be any subset of the coordinates, and assume
\[
  R:= [d]\setminus (\ind(X) \cup S)\neq \emptyset.
\]
Define
\[
  b:=\min\{\,b_i : i\in R\,\}.
\]
Then one has the inequality
\begin{equation}\label{eq:gen-ineq-including}
  \sum_{\ind(X_j)\subseteq \ind(X) \cup S} |X_j|\ \ge\ b\,|X|.
\end{equation}
Equivalently,
\begin{equation}\label{eq:gen-ineq-excluding}
  \sum_{\substack{\ind(X_j)\subseteq \ind(X)\cup S\\ X_j\neq X}} |X_j|\ \ge\ (b-1)\,|X|.
\end{equation}
\end{proposition}

\begin{proof}
Define \(\hat{\mathbf{b}}=(\hat b_1,\dots,\hat b_d)\) by
\[
  \hat b_i :=
  \begin{cases}
    b_i, & i\in \ind(X) \cup S,\\
    b,   & i\in R.
  \end{cases}
\]
Let \(\widehat{\mathcal{P}}:=\mathcal{P}(\hat{\mathbf{b}})\subseteq \mathcal{P}(\mathbf{b})\).
By translating in \(G(\mathbf{b})\),
we may assume \(X\cap \widehat{\mathcal{P}}\neq\emptyset\); then \(\hat X:=X\cap \widehat{\mathcal{P}}\) is a cell in \(\widehat{\mathcal{P}}\),
and because we only shrank coordinates in \([d]\setminus \ind(X)\), we have \(\hat X = X\).

For each \(j\), write \(\hat X_j:=X_j\cap \widehat{\mathcal{P}}\); then the nonempty \(\hat X_j\) form a cell partition of \(\widehat{\mathcal{P}}\).

Now partition \(\widehat{\mathcal{P}}\) into \(b\) congruence classes:
for \(0\le k\le b-1\), set
\[
  \mathcal{G}_k := \Bigl\{\,c=(c_1,\dots,c_d)\in \widehat{\mathcal{P}} : \sum_{i\in R} c_i \equiv k \pmod b\,\Bigr\}.
\]
Then \(\widehat{\mathcal{P}}=\bigsqcup_{k=0}^{b-1}\mathcal{G}_k\) and \(|\mathcal{G}_k|=|\widehat{\mathcal{P}}|/b\) for all \(k\).

Consider the contribution of the cells with \(\ind(X_j)\subseteq \ind(X) \cup S\).
If \(\ind(X_j)\subseteq \ind(X) \cup S\), then \(R\subseteq [d]\setminus \ind(X_j)\), i.e.\ every coordinate in \(R\) is fixed inside \(X_j\),
hence \(\hat X_j\) is contained in a \emph{single} \(\mathcal{G}_k\) (and in that case \(\hat X_j=X_j\) because we shrank only fixed coordinates of \(X_j\)).
On the other hand, if \(\ind(X_j)\nsubseteq \ind(X)\cup S\), then \(\ind(X_j)\) contains some coordinate of \(R\),
so \(\hat X_j\) varies in at least one coordinate in \(R\) of range \(\{0,\dots,b-1\}\); it follows by a direct counting argument that
\(|\hat X_j\cap \mathcal{G}_k|=|\hat X_j|/b\) for every \(k\).

Let
\[
  A_k := \sum_{\substack{\ind(X_j)\subseteq \ind(X)\cup S\\ \hat X_j\subseteq \mathcal{G}_k}} |\hat X_j|
  \qquad(0\le k\le b-1).
\]
Since \(\sum_j |\hat X_j\cap \mathcal{G}_k| = |\mathcal{G}_k|\) for each \(k\), and the cells with \(\ind(X_j)\nsubseteq \ind(X)\cup S\)
contribute the \emph{same} amount to each \(\mathcal{G}_k\) (namely \(|\hat X_j|/b\)), the quantities \(A_k\) must be equal for all \(k\).
Let \(k_0\) be such that \(\hat X\subseteq \mathcal{G}_{k_0}\). Then \(A_{k_0}\ge |\hat X|=|X|\), hence for every \(k\),
\(A_k=A_{k_0}\ge |X|\). Summing over \(k\) gives
\[
  \sum_{\ind(X_j)\subseteq \ind(X)\cup S} |X_j|
  \;\ge\;\sum_{\ind(X_j)\subseteq \ind(X) \cup S} |\hat X_j|
  \;=\;\sum_{k=0}^{b-1} A_k
  \;\ge\; b\,|X|,
\]
which is \eqref{eq:gen-ineq-including}. Subtracting \(|X|\) from both sides yields \eqref{eq:gen-ineq-excluding}.
\end{proof}

\subsection{Two immediate corollaries}

To emphasize how Proposition~\ref{prop:gen-cell-ineq} packages the two inequalities used later, we record them as corollaries.

\begin{corollary}\label{cor:bff-prop1-including}
In the setting of Proposition~\ref{prop:gen-cell-ineq}, take \(S=\emptyset\) and set
\[
  b_1(X):=\min\{\,b_i : i\notin \ind(X)\,\}.
\]
Then
\[
  \sum_{\ind(X_j)\subseteq \ind(X)} |X_j|\ \ge\ b_1(X)\,|X|.
\]
Equivalently,
\[
  \sum_{\substack{\ind(X_j)\subseteq \ind(X)\\ X_j\neq X}} |X_j|\ \ge\ (b_1(X)-1)\,|X|,
\]
which recovers the formulation appearing in \cite[\S2]{BFF}.
\end{corollary}

\begin{corollary}\label{cor:two-level-variant}
In the setting of Proposition~\ref{prop:gen-cell-ineq}, fix \(X\in\mathcal{F}\) and let \(b_1(X)\) be as above.
Define the set of \emph{minimal fixed coordinates}
\[
  F_1(X):=\{\, i\notin \ind(X) : b_i=b_1(X)\,\}\subseteq [d]\setminus \ind(X).
\]
If \([d]\setminus (\ind(X)\cup F_1(X))\neq\emptyset\), define the \emph{next} side length
\[
  b_2(X):=\min\{\,b_i : i\notin \ind(X)\cup F_1(X)\,\}.
\]
Then
\[
  \sum_{\ind(X_j)\subseteq \ind(X)\cup F_1(X)} |X_j|\ \ge\ b_2(X)\,|X|.
\]
Equivalently,
\[
  \sum_{\substack{\ind(X_j)\subseteq \ind(X)\cup F_1(X)\\ X_j\neq X}} |X_j|
  \ \ge\ (b_2(X)-1)\,|X|.
\]
\end{corollary}

\begin{remark}
More generally, Proposition~\ref{prop:gen-cell-ineq} lets one ``declare as harmless'' \emph{any} set \(S\) of coordinates
and forces a lower bound with the \emph{minimum} remaining side length outside \(\ind(X) \cup S\).
This is exactly the structural flexibility needed to keep later inductions uniform when \(p_1=2\),
by absorbing the smallest fixed side lengths into \(S\) and letting the inequality detect the next layer.
\end{remark}

\subsection{Coset-partition formulation of the cell inequalities}

Let \(n=\prod_{j=1}^{\ell} p_j^{s_j}\) with primes \(p_1<\cdots<p_\ell\), and set
\[
\mathbf b=(\underbrace{p_1,\dots,p_1}_{s_1},\underbrace{p_2,\dots,p_2}_{s_2},\dots,\underbrace{p_\ell,\dots,p_\ell}_{s_\ell}).
\]

We recall the parallelotope function in \cite{BFF}.
For \(k\in \mathbb Z_n\), write its residue modulo \(p_j^{s_j}\) in base \(p_j\) as
\[
k \equiv \sum_{r=1}^{s_j} b_r^{(j)} p_j^{\,s_j-r}\pmod{p_j^{s_j}},
\qquad
b_r^{(j)}\in\{0,\dots,p_j-1\}.
\]
Define
\[
\varphi(k):=(\varphi^{(1)}(k),\dots,\varphi^{(\ell)}(k)),
\qquad
\varphi^{(j)}(k):=(b_1^{(j)},\dots,b_{s_j}^{(j)})\in [0,p_j-1]^{s_j}.
\]
Then \(\varphi:\mathbb Z_n\to \mathcal P(\mathbf b)\) is a bijection.
If \(K\subseteq \mathbb Z_n\) is a coset and
\[
|K| = \prod_{j=1}^{\ell} p_j^{r_j},
\]
then \(\varphi(K)\) is a cell with index set
\[
\ind(\varphi(K))
=
\bigcup_{j=1}^{\ell}
\left\{
\sum_{i<j} s_i + 1,\dots,\sum_{i<j} s_i + r_j
\right\},
\]
and \(|\varphi(K)|=|K|\). Conversely, the preimage of every cell of this form is a coset.

Moreover, if \(K, K'\subseteq \mathbb Z_n\) are cosets, then
\begin{equation}
\label{eq:index-inclusion}
  \ind(\varphi(K'))\subseteq \ind(\varphi(K))
  \quad\Longleftrightarrow\quad
  |K'|\mid |K|.
\end{equation}

Let \(\mathcal C=(K_1,\dots,K_t)\) be a coset partition of \(\Z_n\). For a prime \(p\), let \(\nu_p(\cdot)\) denote the \(p\)-adic valuation.

\begin{corollary}\label{cor:coset-bff}
Fix \(K\in\mathcal C\) and put
\[
  p_1(n/|K|):=\min\{\,p \text{ prime} : p\mid n/|K|\,\}.
\]
Then
\[
  \sum_{\substack{K_i\in\mathcal C\\ |K_i|\mid |K|}} |K_i|
  \ \ge\ p_1(n/|K|)\,|K|.
\]
Equivalently,
\[
  \sum_{\substack{K_i\in\mathcal C\\ K_i\neq K,\ |K_i|\mid |K|}} |K_i|
  \ \ge\ (p_1(n/|K|)-1)\,|K|.
\]
\end{corollary}

\begin{corollary}\label{cor:coset-twolevel}
Fix \(K\in\mathcal C\) and let \(p_1(n/|K|)\) be as above.
Assume that \(n/|K|\) has at least two distinct prime divisors, and define
\[
  p_2(n/|K|):=\min\{\,p \text{ prime} : p\mid n/|K|,\ p\neq p_1(n/|K|)\,\},
  \qquad
  e_1(n/|K|):=\nu_{p_1(n/|K|)}(n/|K|).
\]
Then
\[
  \sum_{\substack{K_i\in\mathcal C\\ |K_i|\mid |K|\,p_1(n/|K|)^{e_1(n/|K|)}}} |K_i|
  \ \ge\ p_2(n/|K|)\,|K|.
\]
Equivalently,
\[
  \sum_{\substack{K_i\in\mathcal C\\ K_i\neq K,\ |K_i|\mid |K|\,p_1(n/|K|)^{e_1(n/|K|)}}} |K_i|
  \ \ge\ (p_2(n/|K|)-1)\,|K|.
\]
\end{corollary}

\begin{proof}[Proof of both corollaries]
Apply Corollary~\ref{cor:bff-prop1-including} (resp.\ Corollary~\ref{cor:two-level-variant})
to the cell partition \((\varphi(K_1),\dots,\varphi(K_t))\) of \(\mathcal P(\mathbf b)\).

For Corollary~\ref{cor:coset-bff}, the claim follows from \(|\varphi(K_i)|=|K_i|\),
the identity \(b_1(\varphi(K))=p_1(n/|K|)\), and \eqref{eq:index-inclusion}.

For Corollary~\ref{cor:coset-twolevel}, set \(X:=\varphi(K)\), and let \(Y\) be the cell obtained from \(X\)
by declaring all coordinates in \(F_1(X)\) to be free. Then
\[
\ind(Y)=\ind(X)\cup F_1(X),
\]
and
\[
|Y|
=
|X|\prod_{r\in F_1(X)} b_r
=
|K|\,p_1(n/|K|)^{e_1(n/|K|)}.
\]
Since \(Y\) is a cell, \(K_Y:=\varphi^{-1}(Y)\) is a coset and \(\varphi(K_Y)=Y\).
Hence, for every \(K_i\in\mathcal C\),
\[
\ind(\varphi(K_i))\subseteq \ind(X)\cup F_1(X)
\quad\Longleftrightarrow\quad
\ind(\varphi(K_i))\subseteq \ind(Y)
\quad\Longleftrightarrow\quad
\ind(\varphi(K_i))\subseteq \ind(\varphi(K_Y)).
\]
Applying \eqref{eq:index-inclusion} with \(K_Y\) in place of \(K\), we obtain
\[
\ind(\varphi(K_i))\subseteq \ind(X)\cup F_1(X)
\quad\Longleftrightarrow\quad
|K_i|\mid |K_Y|
=
|Y|
=
|K|\,p_1(n/|K|)^{e_1(n/|K|)}.
\]
This gives the stated inequality.
\end{proof}

\section{Main Result}
In this section, we prove the main theorem by developing an inductive scheme that bounds the prime exponents appearing in the coset sizes. The basic mechanism is a divisor-sum inequality, which controls the largest relevant prime and its exponent once the contribution from smaller divisor classes is bounded.

The case \(p_1=2\) requires a separate treatment, while the case of odd primes is handled uniformly through the divisor-sum inequality above. The proof is then organized through two nested inductions: first, an induction that bounds sums over proper divisor classes, and second, an outer induction on the prime level \(h\), which together yield uniform bounds for all exponents and hence for \(n\) itself.

Throughout this section, we fix a coset partition \(\Cset=(K_1,\dots,K_t)\) of \(\Z_n\) satisfying \eqref{eq:basicascoset} with prime factorization,
\[
n=\prod_{j=1}^{\ell} p_j^{s_j},
\qquad p_1<\cdots<p_\ell.
\]

\subsection{Notation for the proof}
For a positive integer \(N\), \(\sigma(N)\) denotes the sum of all positive divisors of \(N\). We say that a positive integer \(Q\) is supported on a set of primes \(S\) if every prime divisor of \(Q\) lies in \(S\).

\begin{definition}\label{def:symbols}
Let \(\Cset=(K_1,\dots,K_t)\) be a coset partition of \(\Z_n\).
\begin{enumerate}[label=(\roman*)]
\item For a subfamily \(\mathcal{F}\subseteq \Cset\) and \(1\le h\le \ell\), define
\[
d_h(\mathcal{F}) := \max\{\,\nu_{p_h}(|K|): K\in\mathcal{F}\,\}.
\]
\item For \(K\in\Cset\), define
\[
\Dset(K):=\{\,K'\in \Cset: K'\neq K,\ |K'|\mid |K|\,\}.
\]
\item For \(K\in\Cset\), define
\[
\pi_h(K):=\prod_{j=h+1}^{\ell} p_j^{\nu_{p_j}(|K|)}.
\]
\item For \(1\le h\le \ell\), define
\[
\Aset_h
:=\{\,K\in\Cset:|K|\text{ is supported on primes 
}\{p_1,\dots,p_h\}\,\}.
\]
\item Let \(1\le h\le \ell\) and let \(Q\) be supported on primes \(\{p_{h+1},\dots,p_\ell\}\).
Define
\[
\Lset_h(Q)
:=\{K\in\Cset:\pi_h(K) = Q \}.
\]
In particular, \(\Aset_h=\Lset_h(1)\).
\item If \(\mathcal{F}\subseteq \Cset\), we say \(K\in\mathcal{F}\) is \emph{division minimal in \(\mathcal{F}\)} if
\(K'\in\mathcal{F}\) and \(|K'|\mid |K|\) imply \(K'=K\).
\end{enumerate}
\end{definition}

\subsection{A divisor-sum inequality}

\begin{proposition}\label{prop:largest-prime}
Let \(N=\prod_{j=1}^{\ell} p_j^{s_j} > 1\) be a positive integer with primes \( p_1<\cdots<p_\ell \).
Fix a constant \(C_0>0\) and suppose
\begin{equation}\label{eq:assump}
  N-\sigma\!\left(\frac{N}{p_\ell}\right)\le C_0 .
\end{equation}
Then \(p_\ell\) is bounded in terms of \(C_0\) only. Moreover, if \(p_\ell\ge 3\), then each
exponent \(s_j\) (\(j = 1, \ldots,\ell\)) is also bounded in terms of \(C_0\) only.
\end{proposition}

\begin{proof}
Write \(M:=N/p_\ell^{s_\ell}=\prod_{j=1}^{\ell-1}p_j^{s_j}\).
Using \(\sigma(p^{a})=(p^{a+1}-1)/(p-1)\) for a prime \(p \), we get
\[
  \sigma\!\left(\frac{N}{p_\ell}\right)
  =\sigma(p_\ell^{\,s_\ell-1})\,\sigma(M)
  =\frac{p_\ell^{s_\ell}-1}{p_\ell-1}\,\sigma(M)
  < \frac{p_\ell^{s_\ell}}{p_\ell-1}\,\sigma(M).
\]
Hence
\begin{equation}\label{eq:key-lb}
  N-\sigma\!\left(\frac{N}{p_\ell}\right)
  >  N\!\left(1-\frac{\sigma(M)}{(p_\ell-1)M}\right).
\end{equation}

If \(M=1\), then the right hand side of \eqref{eq:key-lb} equals
\(p_\ell^{s_\ell}\bigl(1-\frac{1}{p_\ell-1}\bigr)\), which diverges as \(p_\ell\to\infty\) or
\(s_\ell\to\infty\) if \(p_\ell \ge 3\); thus \eqref{eq:assump} forces both to be bounded.

Assume \(M>1\). Then
\[
  \frac{\sigma(M)}{M}
  \;<\; \prod_{j<\ell}\frac{p_j}{p_j-1}
  \;<\; \prod_{\substack{p\le p_\ell\\ p\ \text{prime}}}\frac{p}{p-1}.
\]
A Mertens--type bound (see \cite{RosserSchoenfeld1962}) implies
\begin{equation}\label{eq:RS}
  \frac{\sigma(M)}{M}
  < e^{\gamma}\,\log p_\ell\!\left(1+\frac{1}{\log^{2}p_\ell}\right).
\end{equation}
Combining \eqref{eq:key-lb}, \eqref{eq:RS}
and the assumption \eqref{eq:assump} gives
\begin{equation}\label{eq:main}
  N\!\left(1- e^{\gamma}\!\left(1+\frac{1}{\log^{2}p_\ell}\right)
           \frac{\log p_\ell}{p_\ell-1}\right) < C_0 .
\end{equation}
The left hand side diverges as \(p_\ell\to\infty\), hence \(p_\ell\) is bounded.

If \(p_\ell\ge5\), then the parenthetical factor admits a positive absolute lower
bound
\[
  \delta_5:=1- e^{\gamma}\!\left(1+\frac{1}{(\log 5)^2}\right)\frac{\log 5}{4}
  \approx 0.0067>0 .
\]
Thus \(N\,\delta_5<C_0\), which forces \(N\) (and hence all exponents \(s_j\)) to
be bounded.

If \(p_\ell = 3\), write \(N = 2^{s_1}3^{s_2}\). Then
\[
 N-\sigma\!\Big(\frac{N}{3}\Big)
  = 2^{s_1} + \frac{3^{s_2}-1}{2}.
\]
Hence, \(s_1\) and \(s_2\) are bounded.
\end{proof}

\subsection{Auxiliary Lemmas}
\begin{proposition}
\label{prop:bound-v2-even}
Assume \(p_1=2\).
Fix an integer \(Q \ge 1\) supported on \(\{p_2,\dots,p_\ell\}\). Suppose there exists a constant \(C(m)>0\) depending only on \(m\) such that
\[ \sum_{\substack{Q'\mid Q\\ Q'\neq Q}}\ \sum_{L\in\Lset_1(Q')} |L|\ \le\ C(m). \]
Then \(d_1(\Lset_1(Q))\) is bounded in terms of \(m\) only.
\end{proposition}

\begin{proof}
We may assume \(\Lset_1(Q)\neq\emptyset\) and
\(d_1(\Lset_1(Q))\ge1\). Choose \(K\in\Lset_1(Q)\) such that
\(\alpha := \nu_2(|K|)=d_1(\Lset_1(Q))\). We distinguish two cases.

\smallskip\noindent\textbf{Case 1: } \(n/|K|\) is not a power of \(2\). If \(K'\in \Lset_1(Q)\setminus\{K\}\), then 
\(\nu_2(|K'|)<\alpha \), so \(|K'|\mid |K|/2\).
We obtain
\[
\sum_{\substack{K' \in \Lset_1(Q) \\ K'\neq K}}\ |K'|
\ \le\
Q\,\sigma\!\left(2^{\alpha-1}\right) + m - 1,
\]
since only singleton cosets repeat \(m\) times by assumption \eqref{eq:basicascoset}.
Therefore, after enlarging \(C(m)\), we have
\begin{equation}
\label{eq:decomposition-v2}
\sum_{K'\in\Dset(K)} |K'|
\le
\sum_{\substack{\pi_1(K')\mid Q \\ K'\neq K}} |K'|
\le
\sum_{\substack{K' \in \Lset_1(Q) \\ K'\neq K}} |K'| +\sum_{\substack{Q'\mid Q\\ Q'\neq Q}}\ \sum_{L\in\Lset_1(Q')} |L|\
\le
Q\,\sigma\!\left(2^{\alpha-1}\right) + C(m).
\end{equation}
Suppose \(n/|K|\) is odd. Then \(p_1(n/|K|) \ge 3\). Applying
 Corollary~\ref{cor:coset-bff}, we have 
\begin{equation}
\label{eq:D(K)-evaluation-v2}
  2^{\alpha+1}Q \le (p_1(n/|K|)-1)\,|K|
  \ \le\
  \sum_{\substack{K' \in \Dset(K)}} |K'|.
\end{equation}
Otherwise, \(n/|K|\) is even and not a power of \(2\). The second smallest prime divisor
\(p_2(n/|K|)\) satisfies \(p_2(n/|K|)\ge3\).
Apply Corollary~\ref{cor:coset-twolevel} with \(p_1(n/|K|)=2\) to obtain
\begin{equation}\label{eq:pi_1(K)-evaluation-v2}
2^{\alpha+1}Q \le
  (p_2(n/|K|)-1)\,|K|
  \ \le\
  \sum_{\substack{\pi_1(K')\mid Q \\ K'\neq K}} |K'|.
\end{equation}
Combining \eqref{eq:D(K)-evaluation-v2}, \eqref{eq:pi_1(K)-evaluation-v2} and \eqref{eq:decomposition-v2}, in either case, we have
\[
2^{\alpha+1}Q\le \sigma\!\left(2^{\alpha-1}\right)Q + C(m)
= (2^\alpha-1)Q +  C(m).
\]
Dividing by \(Q\) yields
\[
2^{\alpha+1}\le (2^\alpha-1)+\frac{C(m)}{Q}
\le 
(2^\alpha-1)+C(m).\]
Hence \(\alpha=d_1(\Lset_1(Q))\) is bounded in terms of \(m\) only.

\smallskip\noindent\textbf{Case 2: } \(n/|K|\) is a power of \(2\).
Then \(Q\) contains all odd prime factors of \(n\), since \(n/|K|\) is a power of \(2\).
Therefore \(\pi_1(K')\mid Q\) for all \(K'\in\Cset\).
Moreover, if \(K'\in \Lset_1(Q)\), then \(|K'|=2^eQ\) for some \(e=\nu_2(|K'|)\).
Since \(Q=n/2^{s_1}\) and \(|K'|\le n/3\) for all \(K'\in\Cset\) by \eqref{eq:basicascoset}, we obtain
\[
2^eQ \le \frac{n}{3}=\frac{2^{s_1}Q}{3},
\]
hence \(e\le s_1-2\).
Therefore,
\[
n = \sum_{K' \in \Cset} |K'|
= \sum_{K' \in \Lset_1(Q)} |K'|
+ \sum_{\substack{Q'\mid Q\\ Q'\neq Q}} \sum_{L\in\Lset_1(Q')} |L|
\le Q\,\sigma\!\left(2^{s_1-2}\right)+C(m).
\]
Dividing by \(Q\) yields
\[
2^{s_1}\le \sigma\!\left(2^{s_1-2}\right)+C(m).
\]
Hence \(s_1\) is bounded in terms of \(m\), and therefore so is \(\alpha\).
\end{proof}

\begin{proposition}
\label{prop:bound-vh_odd}
Assume \(p_h\ge 3\).
Fix \(1\le h\le \ell\) and a positive integer \(Q\) supported on primes \(\{p_{h+1},\dots,p_\ell\}\) and assume \(\Lset_h(Q)\neq\emptyset\) and \(d_h(\Lset_h(Q))\ge 1\). Suppose there exists a constant \(C(m)>0\) depending only on \(m\) such that
\[ \sum_{\substack{Q'\mid Q\\ Q'\neq Q}}\ \sum_{L\in\Lset_h(Q')} |L|\ \le\ C(m). \]
Then \(p_h\) and \(d_h(\Lset_h(Q))\) are bounded in terms of \(m\) only.
\end{proposition}

\begin{proof}
Choose \(K\in\Lset_h(Q)\) such that \(
\nu_{p_h}(|K|)=d_h(\Lset_h(Q))
\)
and \(K\) is division minimal in
\[\{L\in\Lset_h(Q):\nu_{p_h}(|L|)=d_h(\Lset_h(Q))\}.
\]
Write \(N_K:=|K|/Q\). If \(K'\in\Dset(K)\cap\Lset_h(Q)\), then by the division-minimality,
\(\nu_{p_h}(|K'|)<\nu_{p_h}(|K|)\). Indeed if \(\nu_{p_h}(|K'|)=\nu_{p_h}(|K|)\), then \(K'\) would also belong to
\[
\{L\in\Lset_h(Q):\nu_{p_h}(|L|)=d_h(\Lset_h(Q))\},
\]
and \(|K'|\mid |K|\), contradicting the division-minimality of \(K\). This yields \(|K'|\mid |K|/p_h\), and therefore,
\[
\sum_{K'\in\Dset(K)\cap\Lset_h(Q)} |K'|
\ \le\
Q\,\sigma\!\left(\frac{N_K}{p_h}\right) + m - 1
\]
by assumption \eqref{eq:basicascoset}. After enlarging \(C(m)\), we have
\[
\sum_{K'\in\Dset(K)} |K'|
\le
\sum_{K'\in\Dset(K)\cap\Lset_h(Q)} |K'|+ \sum_{\substack{Q'\mid Q\\ Q'\neq Q}}\ \sum_{L\in\Lset_h(Q')} |L|
\le
Q\,\sigma\!\left(\frac{N_K}{p_h}\right)+C(m).
\]
By Corollary~\ref{cor:coset-bff}, we have \(|K|\le \sum_{K'\in\Dset(K)} |K'|\). Thus, we obtain
\[
N_KQ = |K| \le
Q\,\sigma\!\left(\frac{N_K}{p_h}\right)+C(m).
\]
Dividing by \(Q\) yields
\[
N_K\le \sigma\!\left(\frac{N_K}{p_h}\right)+\frac{C(m)}{Q}
\le \sigma\!\left(\frac{N_K}{p_h}\right)+C(m).
\]
Hence
\[
N_K-\sigma\!\left(\frac{N_K}{p_h}\right)\le C(m).
\]
Note that \(N_K\) is supported on \(\{p_1,\dots,p_h\}\),
and \(p_h\mid N_K\) because \(d_h(\Lset_h(Q))\ge1\).
Hence \(p_h\) is the largest prime divisor of \(N_K\) and \(\nu_{p_h}(N_K)=\nu_{p_h}(|K|)=d_h(\Lset_h(Q))\).
Applying Proposition~\ref{prop:largest-prime} with \(C_0=C(m)\),
we obtain that \(p_h\) and \(d_h(\Lset_h(Q))\) are bounded in terms of \(m\) only.
\end{proof}

\begin{lemma}\label{lem:bound-vh-general}
Fix \(1\le h\le \ell\) and a positive integer \(Q\) supported on primes \(\{p_{h+1},\dots,p_\ell\}\) and assume \(\Lset_h(Q)\neq\emptyset\) and \(d_h(\Lset_h(Q))\ge 1\). Suppose there exists a constant \(C(m)>0\) depending only on \(m\) such that
\[ \sum_{\substack{Q'\mid Q\\ Q'\neq Q}}\ \sum_{L\in\Lset_h(Q')} |L|\ \le\ C(m). \]
Then \(p_h\) and \(d_h(\Lset_h(Q))\) are bounded in terms of \(m\) only.
\end{lemma}

\begin{proof}
If \(h = 1\) and \(p_h = 2\), then the lemma is a direct consequence of Proposition~\ref{prop:bound-v2-even}. If \(p_h \ge 3\), Proposition~\ref{prop:bound-vh_odd} proves the lemma.
\end{proof}

\begin{lemma}\label{lem:proper-divisor-sum-bound}
Fix \(1\le h\le \ell\) and let \(K\in\Cset\).
Put
\(
Q:=\pi_h(K)\) so that \( K\in\Lset_h(Q).\)
Assume:
\begin{enumerate}[label=(H\arabic*)]
\item \(Q\) is bounded in terms of \(m\) only;
\item \(\displaystyle \sum_{L\in\Aset_h}|L|\) is bounded in terms of \(m\) only.
\end{enumerate}
Then there exists \(C(m)>0\) depending only on \(m\) such that
\[ \sum_{\substack{Q'\mid Q\\ Q'\neq Q}}\ \sum_{L\in\Lset_h(Q')} |L|\ \le\ C(m). \]
\end{lemma}

\begin{proof}
We prove for each divisor \(Q'\mid Q\),
\[
S_h(Q'):=\sum_{L\in\Lset_h(Q')} |L|
\]
is bounded in terms of \(m\) only.
The lemma follows because \(Q\) has only finitely many divisors by assumption (H1) (bounded in terms of \(m\)).

We prove boundedness of \(S_h(Q')\) by induction on \(h\).

\smallskip\noindent\textbf{Base case: \(h=1\).}
If \(Q'=1\), then \(S_1(1)=\sum_{L\in\Aset_1}|L|\) is bounded by assumption (H2).
Fix \(Q'>1\) and assume by induction on divisibility that \(S_1(Q'')\) is bounded for all proper divisors \(Q''\mid Q'\).
If \(d_1(\Lset_1(Q'))=0\), then \(\Lset_1(Q') = \{L \in \Cset \mid |L| = Q'\}\), and this is a singleton set by \eqref{eq:basicascoset}. Hence \(S_1(Q') = Q'\) is bounded. Otherwise \(d_1(\Lset_1(Q'))\ge1\). By the induction hypothesis on divisibility, the hypotheses of Lemma~\ref{lem:bound-vh-general} are satisfied with \(h=1\), so \(d_1(\Lset_1(Q'))\) is bounded.
Note that all sizes in \(\Lset_1(Q')\) are among \(\{Q'p_1^e:0\le e\le d_1(\Lset_1(Q'))\}\),
and each size occurs at most once by \eqref{eq:basicascoset}. Hence, \(S_1(Q')\) is bounded. This completes the base \(h=1\).

\smallskip\noindent\textbf{Induction step.}
Assume the statement holds for \(h-1\ge1\).
Fix \(h\) and \(Q'\mid Q\). If \(Q'=1\), we are done by (H2).
Assume \(Q'>1\) and that \(S_h(Q'')\) is bounded for all proper divisors \(Q''\mid Q'\).
If \(d_h(\Lset_h(Q'))=0\), then \(\Lset_h(Q')\subseteq \Lset_{h-1}(Q')\) and we conclude by the induction hypothesis on \(h-1\).
Now we assume \(d_h(\Lset_h(Q'))\ge1\). By the induction hypothesis on divisibility,
Lemma~\ref{lem:bound-vh-general} gives bounds on \(p_h\) and \(d_h(\Lset_h(Q'))\).

For each \(0\le e\le d_h(\Lset_h(Q'))\), set \(Q'_e:=Q'p_h^e\).
If \(L\in\Lset_h(Q')\) and \(\nu_{p_h}(|L|)=e\), then \(L\in\Lset_{h-1}(Q'_e)\) by definition.
Thus
\[
S_h(Q')\le \sum_{e=0}^{d_h(\Lset_h(Q'))} S_{h-1}(Q'_e).
\]
Each \(Q'_e\) is bounded in terms of \(m\), and \(\sum_{L\in\Aset_{h-1}}|L|\) is bounded because \(\Aset_{h-1}\subseteq\Aset_h\).
Hence each \(S_{h-1}(Q'_e)\) is bounded by the induction hypothesis on \(h-1\), and so \(S_h(Q')\) is bounded.
This completes the induction.
\end{proof}

\subsection{Proof of the theorem}

\begin{theorem}\label{thm:main}
Assume \(n=\prod_{j=1}^{\ell} p_j^{s_j}\) with primes \( p_1<\cdots<p_\ell\).
Let \(\Cset=(K_1,\dots,K_t)\) be a coset partition of \(\Z_n\) satisfying \eqref{eq:basicascoset}.
Then, for every \(h=1,\dots,\ell\),
\[
\sum_{K\in\Aset_h}|K|
\]
is bounded by a constant depending only on \(m\).
\end{theorem}

\begin{proof}
We argue by induction on \(h\).

\smallskip\noindent\textbf{Base case: \(h=1\).}
If \(d_1(\Aset_1)=0\), then every \(K\in\Aset_1\) has \(|K|=1\), and there are at most \(m\) such cosets, so the sum is bounded.
Assume \(d_1(\Aset_1)\ge1\).
Then \(\Lset_1(1)\neq\emptyset\), and since \(Q=1\) has no proper divisors,
the hypothesis of Lemma~\ref{lem:bound-vh-general} with \(h=1\) and \(Q = 1\) is trivially satisfied.
Hence \(d_1(\Lset_1(1))=d_1(\Aset_1)\) is bounded in terms of \(m\).
Hence all sizes in \(\Aset_1\) are among \(\{p_1^e:0\le e\le d_1(\Aset_1)\}\) and each size occurs at most once, except that \(|K|=1\) can repeat at most \(m\) times.
Thus \(\sum_{K\in\Aset_1}|K|\) is bounded.

\smallskip\noindent\textbf{Induction step.}
Fix \(h\ge2\), and assume \(\sum_{K\in\Aset_j}|K|\) is bounded for every \(j<h\).
We show the same for \(\Aset_h\).

First, we bound \(p_h\) and \(d_h(\Aset_h)\). If \(d_h(\Aset_h)=0\) there is nothing to prove about the \(p_h\)-part.
Assume \(d_h(\Aset_h)\ge1\).
Because \(Q=1\) has no proper divisors, the hypothesis of Lemma~\ref{lem:bound-vh-general}
is trivially satisfied, and the lemma yields bounds on \(p_h\) and \(d_h(\Aset_h)\).

Next, we bound \(d_{h-1}(\Aset_h),d_{h-2}(\Aset_h),\dots,d_1(\Aset_h)\) by descending induction on the index.
Let \(1\le j<h\) and assume that \(d_{j+1}(\Aset_h),\dots,d_h(\Aset_h)\) are already known to be bounded.
If \(d_j(\Aset_h)=0\) we are done for this \(j\).
Otherwise choose \(K\in\Aset_h\) such that \(\nu_{p_j}(|K|) = d_j(\Aset_h)\) and set
\(
Q:=\pi_j(K).
\)
Because \(K\in\Aset_h\), only primes \(p_{j+1},\dots,p_h\) can appear in \(Q\).
Their exponents are bounded by the descending induction hypothesis, and the primes themselves are bounded
because \(p_h\) has already been shown to be bounded. Hence \(Q\) is bounded in terms of \(m\).
Moreover \(K\in\Lset_j(Q)\).

Since \(j<h\), the outer induction hypothesis gives that \(\sum_{L\in\Aset_j}|L|\) is bounded,
so Lemma~\ref{lem:proper-divisor-sum-bound} (applied with this \(j\) and \(K\)) provides a bound for
\(\sum_{\substack{Q'\mid Q\\ Q'\neq Q}}\ \sum_{L\in\Lset_j(Q')} |L|\) depending only on \(m\).
Then Lemma~\ref{lem:bound-vh-general} (with \(h=j\)) shows that \(d_j(\Lset_j(Q))\) is bounded. 
By choice of \(Q\), we note that \(d_j(\Aset_h)\le d_j(\Lset_j(Q))\), and hence \(d_j(\Aset_h)\) is bounded.

Thus \(d_1(\Aset_h),\dots,d_h(\Aset_h)\) are all bounded.
Consequently every size \(|K|\) with \(K\in\Aset_h\) divides
\[
M_h:=\prod_{j=1}^{h} p_j^{d_j(\Aset_h)}.
\]
Observe that \(M_h\) is bounded since \(p_h\) and each exponent are already bounded. By \eqref{eq:basicascoset}, each size \(>1\) occurs at most once, and \(|K|=1\) occurs at most \(m\) times, so
\[
\sum_{K\in\Aset_h}|K|
\le
(m-1)+\sum_{d\mid M_h} d
=
(m-1)+\sigma(M_h),
\]
which is bounded in terms of \(m\) only.
\end{proof}

\begin{corollary}\label{cor:main-coset}
Under the assumptions of Theorem~\ref{thm:main}, the integer \(n\) is bounded in terms of \(m\) only.
\end{corollary}

\begin{proof}
Since \(\Aset_\ell=\Cset\) and \(\sum_{K\in\Cset}|K|=|\Z_n|=n\),
the claim follows by taking \(h=\ell\) in Theorem~\ref{thm:main}.
\end{proof}

\section{Discussion}

Corollary~\ref{cor:main-coset} shows that, for each fixed multiplicity \(m\), the largest modulus
\[
n:=\mathrm{lcm}(n_1,\dots,n_t)
\]
of a DCS satisfying \eqref{eq:basicas} is bounded in terms of \(m\) alone. 
Our proof, however, is purely qualitative. 
Because the argument proceeds through nested inductions on prime layers and divisor classes, any quantitative upper bound extracted from the proof would likely be extremely large and far from optimal.

By contrast, the computational data of Ekhad et al.~\cite{EFZ} suggest that the actual extremal behavior is much more rigid. 
In particular, among the examples listed there for \(m\le 32\), the minimal candidate for the optimal upper bound appears always to be of the form
\[
2^{s_1}3^{s_2}.
\]
This leads naturally to the following problem.

\begin{problem}
For each fixed \(m\), determine the optimal upper bound for the largest modulus \(n\) among DCSs satisfying \eqref{eq:basicas}.
\end{problem}

The available data suggest the following conjectural principle.

\begin{conjecture}
Fix \(m\). Among DCSs satisfying \eqref{eq:basicas},
the maximum value of the largest modulus is attained by one for which
\[
n=2^{s_1}3^{s_2}.
\]
\end{conjecture}

If true, this would reduce the extremal problem to the \(2\)-\(3\)-primary case. 
Moreover, the methods developed in the present paper suggest that, in this special case, the exponents \(s_1\) and \(s_2\) should admit comparatively small bounds. 
It is therefore plausible that these bounds lie below the range already verified computationally in~\cite{EFZ}. 
Such a result would bring one closer to a complete characterization of the examples appearing in the tables of~\cite{EFZ}.

A second issue concerns the role of disjointness in density estimates for uncovered sets. 
As explained in Remark~\ref{rem:balister-simpson}, removing the \(m\) copies of the repeated largest modulus from a DCS leaves a distinct-moduli covering problem whose uncovered set has density
\[
\delta(R)=\frac{m}{n}.
\]
Thus, in the present setting, boundedness of 
\(n\) is closely related to a lower bound for the density of the uncovered set. This motivates the following broader question.

\begin{question}
To what extent does the additional structure coming from disjointness improve lower bounds for the density of the uncovered set in distinct-moduli covering problems?
\end{question}

The arguments developed in this paper suggest that disjointness is not merely a technical assumption. 
Rather, it imposes strong structural constraints, and these constraints may force stronger boundedness phenomena than those presently available in the general distinct-moduli setting considered by Balister et al. 
A natural direction for future work would be to identify a quantitative measure of ``distance from disjointness'' and to study how lower bounds for uncovered density improve as this parameter decreases. 
A satisfactory formulation of such a principle would likely clarify the relationship between the present method and the framework of~\cite{BBMST}, and may also shed further light on the extremal problem discussed above.

\section*{Declaration of generative AI and AI-assisted technologies in the research and manuscript preparation process}

During the preparation of this work, conversations with ChatGPT (OpenAI) contributed substantially to the development of the research. In particular, ChatGPT assisted in developing the proof of Proposition 2.1 based on an idea proposed by the author. The divisor-induction argument in Lemma 3.6 was suggested by ChatGPT, although the author had initially developed an inductive argument based on individual prime exponents.

ChatGPT was also used to obtain feedback on the exposition, including clarity, conciseness, possible redundancies, and possible issues in the presentation of arguments, as well as to assist in drafting limited portions of the manuscript and correcting typographical and grammatical errors. All AI-generated mathematical arguments and suggestions were subsequently reviewed and independently verified by the author, who takes full responsibility for the mathematical content of the manuscript.

\bibliographystyle{plain}
\bibliography{references}

\end{document}